%% file: main.tex
\documentclass[11pt]{article}
\usepackage{amssymb}
\usepackage{multicol}
\usepackage{graphicx}
\usepackage{amsthm}
\usepackage{verbatim}

\input{macros_short.tex}

\usepackage{amsmath, epsfig}
\newtheorem{theorem}{Theorem}
\newtheorem{lemma}{Lemma}

\newcommand{\e}{{\mathbb{E}}}

\newcommand{\Var}{{\rm Var}}

\newcommand{\p}{\mathbb{P}}

\renewcommand{\phi}{\varphi}

\newcommand{\ignore}[1]{}
\title{Spreading Processes and Large Components in Ordered, Directed
Random Graphs}
\author{
Paul Horn\thanks{Department of Mathematics, Harvard University, {\tt phorn@math.harvard.edu}} \and Malik Magdon-Ismail
\thanks{Department of Computer Science, Rensselaer Polytechnic Institute, 
{\tt magdon@cs.rpi.edu}}}
\begin{document}
\maketitle

\begin{abstract}%
Order the vertices of a directed
random graph \math{v_1,\ldots,v_n};
edge \math{(v_i,v_j)} for \math{i<j} exists independently
 with probability \math{p}.
This random graph model is related to certain spreading processes on 
networks. We consider the component reachable from 
\math{v_1} and prove existence of a sharp threshold 
\math{p^*=\log n/n} at which this reachable component transitions from
\math{o(n)} to \math{\Omega(n)}.
\end{abstract}

\section{Introduction}

In this note we study a random graph model that captures the
dynamics of a particular type of spreading process.  
Consider a set of $n$ ordered vertices $\{v_1, \dots, v_n\}$ with vertex \math{v_1}
initially `infiltrated' (at time step 1). At time steps  
\math{2,3,\ldots,n}, vertex
$v_1$ attempts to independently 
infiltrate, with 
probability $p$, each of $v_2,v_3 \dots, v_n$ in turn (one per step). 
Either \math{v_i} gets infiltrated or immunized.  
If $v_i$ is infected, it attempts to infect $v_{i+1},
\dots, v_{n}$, also each with probability $p$;  $v_i$ does not 
attempt to infect $v_{1}, \dots, v_{i-1}$, however, as prior vertices 
are already 
either infiltrated or immunized. 
At time step \math{i}, all infiltrated
vertices \math{v_j} with \math{j<i} are attempting to
infiltrate \math{v_i}, and \math{v_i} gets infiltrated if any
 one of these attempts
succeeds.
Intuitively, \math{v_i} is more likely to get infiltrated if
more vertices are already infiltrated at the time that
\math{v_i} becomes 'succeptible'.  One example of such a contagion process
is given in \cite{malik171}.

This spreading  
process is equivalent to the following random model of an ordered,
directed graph \math{G}: order the vertices \math{v_1,\ldots,v_n}, and for 
\math{i<j}, the directed
edge \math{(v_i,v_j)} exists in \math{G}
 with probability
\math{p} (independently). Vertex \math{v_i} is infected if there is a (directed) path from 
\math{v_1} to \math{v_i}.
The question we address
is, ``What is the size of the set of vertices reachable from $v_1$?'' 
(the size of
the infection).  
We prove the following sharp result.

\begin{theorem} \label{t1}
Let $\mathcal{R}$ be the set of vertices reachable from $v_1$, and suppose $p = \frac{c \log n}{n} + \xi(n)$, where
$\xi(n) = o(\frac{\log n}{n})$ and $c > 0$ is fixed.  Then:
\begin{enumerate}
\item If $c < 1$, then $|\mathcal{R}| = n^{c+o(1)}$, a.a.s.
\item If $c = 1$, then $|\mathcal{R}| = o(n)$, a.a.s.
\item If $c > 1$, then $|\mathcal{R}| = \left(1-\frac{1}{c} + o(1)\right) n$, a.a.s.
\end{enumerate}
\end{theorem}
Recall that 
an event holds a.a.s.(asymptotically almost surely), if it holds with probability $1-o(1)$; that is it holds with probability
tending to one as $n$ tends to infinity.  Note that we do not explicitly care whether $\xi(n)$ is positive or negative in the results above.  

Similar phase transitions are well known for various graph properties in other random graph models.  As shown by Erd\H{o}s and R\'enyi in \cite{er}, in the $G(n,M)$ model of
random graphs, where a graph is chosen independently from all graphs with $M$ edges, there is a similar emergence of a component of size $\Theta(n)$ around $M = \frac{n}{2}$ edges.    
Likewise, a threshold for connectivity was shown for $M = \frac{n \log n}{2}$ edges.  For the more familiar $G(n,p)$ model, where edges are present 
independenty with probability $p$, this translates into a threshold
at $p=\frac{1}{n}$ for a giant component, and at $p=\frac{\log n}{n}$ for
connectivity.  A much more comprehensive account of results on properties of random graphs can be found in \cite{B}.   {\L}uczak in \cite{L} and more recently
{\L}uczak and Seierstad in \cite{LS}, studied the emergence of the giant component in a random directed graphs, in both the directed model where $M$ random edges are present and in the model where edges are present with 
probability $p$.  Thresholds for strong connectivity were established for random directed graphs by Pal\'asti \cite{P} (for random directed graphs with $M$ edges) and Graham and Pike \cite{GP} (for random directed graphs with edge probability $p$).
We are not aware of any results for ordered directed random graphs where 
edges connect vertices of lower index to higher index.


\section{A Proof of Theorem \ref{t1}}

\begin{proof}[Upper bounds:]
For \math{i>1},
let $\mathcal{R}_i$ denote the event that $v_i$ is reachable, 
and let $X_i$ denote the number of paths to vertex $v_i$ in $G$. 
 If $\mathcal{P}_i$ denotes the set of all potential
paths from  \math{v_1} to $v_i$, 
then $X_i = \sum_{x \in \mathcal{P}_i} I(x)$ where $I(x)$ is a $\{0,1\}$ 
indicator random variable indicating whether the path \math{x}
exists in \math{G};
\math{I(x)=1} if and only if all edges in the path $x$ are present
in \math{G}.  
Then,
\begin{align*}
\p(\mathcal{R}_i) = \p(X_i \geq 1) \leq \e[X_i] 
&=\sum_{x \in \mathcal{P}_i} \e[I(x)]\\
&= \sum_{\ell=0}^{i-2} \sum_{\substack{x \in \mathcal{P}_i\\|x|=\ell+1}}   
\e[I(x)] \\
&= \sum_{\ell=0}^{i-2} {i-2 \choose \ell} p^{\ell+1} = p (1+p)^{i-2} \leq pe^{pi}.  
\end{align*}
Let $X$ denote the number of reachable vertices (other than
\math{v_1}).  
\begin{align*}
\e[X] = \sum_{i=2}^{n} \p(\mathcal{R}_{i}) \leq \sum_{i=1}^{n} pe^{pi} &= p\cdot \frac{e^{p(n+1)} - 1}{e^{p}-1}.
\end{align*}

For \math{p= \frac{c\log n}{n} + \xi(n)} with $c < 1$,
\begin{align*}
e^{p(n+1)}-1 &=\exp\left(c\log n + o(\log n)\right)-1 = n^{c + o(1)}, 
\end{align*} and 
\[ \frac{p}{(e^{p}-1)}=\left(\sum_{k=1}^{\infty} \frac{p^{k-1}}{k!}\right)^{-1}
=1+O(p).
\] Thus,
\mand{\e[X]\le  n^{c + o(1)}.} 
Applying Markov's inequality yields that $\p(X > \log(n)\e[X]) =o(1)$,
so \math{X\le \log(n)\e[X] = n^{c+o(1)}}, a.a.s.

Now consider $c > 1$.  Let 
\begin{align*}
\xi'(n) &:= \frac{3}{c} \max\left\{ \frac{n}{\log \log n}, \frac{n^2 \xi(n)}{c\log n} \right\} \\
t &:= \frac{n}{c} - \xi'(n)
\end{align*}
Note that by our choice of $\xi'(n)$, and the fact that $\xi(n) = 
o (\frac{\log n}{n})$, that $\xi'(n) = o(n)$.  
Then, 
\begin{align*}
\p(\mathcal{R}_t) \leq pe^{pt}&= p\exp\left(\left(c\frac{\log n}{n}  + \xi(n)\right)\left(\frac{n}{c} - \xi'(n) \right) \right) \\
&= p \exp\left(\log(n) + \frac{n\xi(n)}{c} - \frac{c (\log n) \xi'(n)}{n} - \xi(n)\xi'(n) \right)  \\
&\leq (1 + o(1))c \exp\left(\log \log(n) + \frac{n \xi(n)}{c} - \frac{c (\log n) \xi'(n)}{n} \right) = o(1) \\
\end{align*} 
Here, the last inequality comes from the fact that, by our choice of $\xi'(n)$,
\[
\frac{c (\log(n)) \xi'(n)}{n}  - \log \log(n) - \frac{n \xi(n)}{c} \geq \frac{1}{3}\xi'(n).
\]

Since \math{pe^{pi}} is increasing in \math{i},
the expected number of reachable vertices
$v_i$ with $i \leq t$ is at most $t\p(\mathcal{R}_t)=o(n)$.  
Applying  Markov's inequality,
$|\mathcal{R}  \cap \{v_1, \dots, v_t\}| = o(n)$ a.a.s.  
Thus, 
$$|\mathcal{R}| \leq n-t+|\mathcal{R}  \cap \{v_1, \dots, v_t\}|
=\textstyle\left(1 - \frac{1}{c} + o(1)\right)n \text{ a.a.s.}$$

For $p = \frac{\log n}{n} + \xi(n)$ with $\xi(n) = 
o\left(\frac{\log n}{n}\right)$, we will write  
$\xi(n) = \omega(n)\frac{\log n}{n}$,
where \math{\omega(n)\rightarrow0}.
Let $t = n\cdot\left(1 - \omega(n) - \frac{1}{\log \log n}\right).$
Then,
\begin{align*}
\p(\mathcal{R}_t) \leq pe^{-pt} 
&= 
\exp\left[(1 + \omega(n))\left({\textstyle 
1-\omega(n)-\frac{1}{\log\log n}}\right)\log n +
\log\left({\textstyle(1+\omega(n))\frac{\log n}{n}}\right)\right]
\\
&= \exp\left[-\omega(n)^2\log n - (1+\omega(n))\frac{\log n}{\log\log n} + 
\log\log n+\log(1+\omega(n))\right]\\
&= o(1),
\end{align*}
Thus the expected value of \math{|\mathcal{R}  \cap \{v_1, \dots, v_t\}|}
is \math{o(n)} and by Markov's inequality, this is also true a.a.s.
Now, since \math{n-t} is also \math{o(n)}, we have that 
\math{R=o(n)} a.a.s.
\end{proof}

To prove the lower bounds, we require a simple lemma similar to 
Dirichlet's theorem.  
Let $d(i)$ denote the number of divisors of $i$ and 
let $d_t(i)$ denote the number of 
divisors of $i$ that are at most~$t$.  
Dirichlet's Theorem states that
\[
\sum_{i=1}^{k} d(i) = k\log k + (2\gamma-1)k + O(\sqrt{k}),
\]
where $\gamma$ is Euler's constant.  For our purposes, 
we need a refinement of this result, summing 
\math{d_t(i)}.
\begin{lemma}\label{l1} 
\math{\displaystyle
\sum_{i=1}^{k} d_t(i) = k \log \min(t,k) + O(k).
}
\end{lemma}
\begin{proof} For \math{t>k} the result follows from Dirichlet's theorem
as we may replace \math{d_t(i)} with \math{d(i)} in the summation.
For \math{t\le k},
\mand{
\sum_{i=1}^{k} d_t(i) = k + \left\lfloor \frac{k}{2} \right\rfloor + \left\lfloor \frac{k}{3} \right\rfloor + \dots + \left\lfloor \frac{k}{t} \right\rfloor\le k\cl H_t,
} 
where \math{\cl H_t} is the \math{t}-th harmonic number.
\end{proof}
\begin{proof}[Lower bounds:]
For exposition, assume that we construct our graph on 
countably many vertices and that we then restrict our attention
to the first $n$ vertices. 
Let \math{X_i} denote the index of the \math{i}-th reachable vertex (that is 
not \math{v_1}).  
If  $X_i > n$ then $|\mathcal{R}| \le i$.  Set
$X_{0} = 1$, and for \math{i\ge 1}, 
$X_{i} - X_{i-1}$ is geometrically
distributed with parameter $1 - (1-p)^{i}$.  
Fix $t$, and consider \math{\e[X_t]}:
\[
\e[X_t] = \sum_{k=1}^{t} \e[X_{k} - X_{k-1}] = \sum_{k=1}^{t} \frac{1}{1 - (1-p)^{k}}.  
\]
Each term is an infinite geometric series, and so
\[
\e[X_t] = \sum_{k=1}^{t} \sum_{j=0}^\infty (1-p)^{kj}.
\]
As this series is absolutely summable (as $\e[X_t]$ is clearly finite), 
Fubini's theorem allows us to rearrange terms in the summation to get
\[
\e[X_t] =t+
\sum_{k=1}^{t} \sum_{j=1}^\infty (1-p)^{kj}= 
t + \sum_{i=1}^{\infty} d_t(i) (1-p)^{i}.  
\]
because the term \math{(1-p)^{i}} appears in the original
summation (where \math{i=kj}) once for every divisor~\math{i} has that is 
at most \math{t}. 
We now use summation by parts to manipulate the second term:
\begin{align*}
\sum_{i=1}^{\infty} d_t(i) (1-p)^{i} &= p \sum_{i=1}^{\infty} (1-p)^{i-1} \left( \sum_{\ell=1}^i d_t(\ell) \right)\\ 
&= p \sum_{i=1}^{\infty} (1-p)^{i-1} (i \log(\min\{t,i\}) + O(i)) \\
  &\le p (\log t + O(1)) \sum_{i=1}^{\infty} i(1-p)^{i-1}.
\end{align*}
Since \math{\sum_{i=1}^{\infty} i(1-p)^{i-1}=1/p^2}, we have that 
\mld{
\e[X_t]= t + \frac{\log t}{p} + O\left(\frac{1}{p}\right).
\label{eq:Et}
}
Furthermore, since \math{X_{k+1}-X_{k}} and \math{X_{k}-X_{k-1}} are
independent,
\begin{align}
\Var(X_t) &= \sum_{k=1}^{t} \frac{p}{(1-(1-p)^{k})^2}\nonumber\\
 &\leq \sqrt{ \left( \sum_{k=1}^{t} \frac{p^2}{(1-(1-p)^k)^{3}} \right) \left( \sum_{k=1}^{t} \frac{1}{(1-(1-p)^k)} \right)} \nonumber \\
&\leq \sqrt{ \frac{t}{p} \e[X_t]}. \label{eqn1}
\end{align}
Here, the first 
inequality follows from an application of Cauchy-Schwarz, and the second 
from  
\math{
\frac{p^2}{(1 - (1-p)^{k})^3} \leq \frac{p^2}{p^3} = \frac{1}{p}. 
}

Now, suppose that $p = c\frac{\log n}{n}+ \xi(n)$ for $c < 1$, 
and set $t = n^{c}\exp(-n|\xi(n)|- \log\log(n))$.  Then, from \r{eq:Et},
\begin{align}
\e[X_t] &\leq n^{c}\exp(-n|\xi(n)|) + \frac{c \log n - 2n|\xi(n)|-\log \log n}{n^{-1}(c \log n + n\xi(n))} + O\left(\frac{\log n}{n}\right) \\
&\leq n^{c}\exp(-n|\xi(n)|) + n - \frac{n^2|\xi(n)|-\log \log n}{(c \log n + n\xi(n))} + O\left(\frac{\log n}{n}\right) \\
&= n\left(1 - \frac{n|\xi(n)|+\log \log(n)}{(c \log n + n \xi(n))}+ o\left({\frac{n|\xi(n)| + \log \log n}{\log n}}\right)\right).
\label{eq:Et1}
\end{align}
For $n$ sufficiently large, \math{\e[X_t]\le n\left(1-\frac{n|\xi(n)|+\log\log n}{2c\log n}\right)}.  
Meanwhile, from $(\ref{eqn1})$, 
\[
\Var(X_t)\le (1+o(1))\sqrt{\frac{n^c}{\log n}\cdot (1+o(1))\frac{n}{c\log n}\cdot\e[X_t]}
= \frac{n^{\frac12(1+c)}}{\log n}\sqrt{\frac{\e[X_t]}{c}}=O\left(\frac{n^{3/2}}{\log n}\right),
\]
because \math{\e[X_t]=O(n)} and \math{c<1}.
Chebyshev's inequality asserts that
\[
\p\left[|X_t - \e[X_t]|  \geq \frac{n^2|\xi(n)| + n\log \log n}{2c\log n} \right] \leq 
\frac{4c^2\log^2 n\cdot \Var(X_t)}{(n^2|\xi(n)|+n(\log\log n))^2} = o(1). 
\]
Thus, $\p\left[X_t \le \e[X_t]+\frac{n\log \log n}{2c \log n}\right] = 1-o(1)$.
Using \r{eq:Et1},
\mand{
\p\left[X_t \le n\left(1 - \frac{\log \log n}{2c\log n}+
o\left({\frac{\log \log n}{c\log n}}\right)\right)\right]=1-o(1),}
i.e., \math{X_t<n} a.a.s. Since \math{X_t<n} implies
\math{|\mathcal{R}| \ge t}, we have that
$|\mathcal{R}| > n^{c}\exp(-n|\xi(n)|-\log\log(n))=n^{c+o(1)}$ a.a.s.  

For $c > 1$, take $t= \frac{n \log\log n}{\log n}$. Then,
using \r{eq:Et}, 
\[
\e[X_t] \leq  \frac{n}{c} + o(n).  
\]
Again, by $(\ref{eqn1})$ and because \math{\e[X_t]=O(n)},
 $\Var(X_t) = O(n^{3/2}\sqrt{\log\log n}/\log n)=o(n^{3/2})$.  
Chebyschev's inequality asserts that
\[
\p\left[|X_t - \e[X_t]| \geq n^{3/4}\right] \le \frac{o(n^{3/2})}{n^{3/2}}
=o(1).
\]
Hence,
\begin{align}
\p\left[X_t \le\e[X_t] +n^{3/4}\right] = 1-o(1).  \label{eqn2}
\end{align}
So, \math{X_t\le \frac{n}{c}+o(n)} a.a.s.
We now consider the vertices indexed higher than \math{X_t} and show
that essentially all of them are reachable.
Let $Y$ be the vertices with index higher than $X_t$ 
which are \emph{not} adjacent to one of the first $t$ reachable 
vertices in \math{v_1,\ldots,v_{X_t}}.
Then
\begin{align*}
\e[|Y|] &= \sum_{j=X_t+1}^{n} (1-p)^{t}
=
(n-X_t)(1-p)^t
\le ne^{-pt}
=
\frac{n}{\log^{c+o(1)}n}=o(n).
\end{align*}
Applying Markov's inequality, $|Y| = o(n)$ with probability $1-o(1)$.  
Since the set of vertices indexed above \math{X_t} that is not reachable is 
a subset of 
\math{Y},  
$|\mathcal{R}| \geq t + (n - X_t)-|Y|$. Since 
\math{|Y|,t} are \math{o(n)} and \math{X_t=\frac{n}{c}+o(n)}, we have
that  $|\mathcal{R}| \ge n(1-\frac1c+o(1))$ with probability $1-o(1)$,
as desired.  
\end{proof}

\paragraph{Acknowledgement.}
Magdon-Ismail acknowledges that 
this research was sponsored by the Army Research Laboratory
and was accomplished under Cooperative Agreement
Number W911NF-09-2-0053. The views and conclusions
contained in this document are those of the authors and
should not be interpreted as representing the official policies,
either expressed or implied, of the Army Research Laboratory
or the U.S. Government. The U.S. Government is
authorized to reproduce and distribute reprints for Government
purposes notwithstanding any copyright notation here
on.

\bibliographystyle{abbrv}
\bibliography{main}

\end{document}

%% file: macros_short.tex
\def\math#1{$#1$}

\def\mand#1{$$#1$$}

\def\frac#1#2{{#1\over #2}}

\def\mld#1{\begin{equation}
#1
\end{equation}}

\DeclareSymbolFont{AMSb}{U}{msb}{m}{n}
\DeclareMathSymbol{\N}{\mathbin}{AMSb}{"4E}
\DeclareMathSymbol{\Z}{\mathbin}{AMSb}{"5A}
\DeclareMathSymbol{\Q}{\mathbin}{AMSb}{"51}
\DeclareMathSymbol{\I}{\mathbin}{AMSb}{"49}
\DeclareMathSymbol{\C}{\mathbin}{AMSb}{"43}

\def\cl#1{{\cal #1}}

\def\exp#1{{\left\langle#1\right\rangle}}

\def\r#1{{(\ref{#1})}}

\def\dotfil{\leaders\hbox to 1.5mm{.}\hfill}
\newcounter{rmnum}
\def\RN#1{\setcounter{rmnum}{#1}\uppercase\expandafter{\romannumeral\value{rmnum}}}
\def\rn#1{\setcounter{rmnum}{#1}\expandafter{\romannumeral\value{rmnum}}}